\documentclass[12pt, reqno]{amsart}

\usepackage{amssymb} 
\usepackage{mathrsfs}
\usepackage{url}

\newtheorem{theorem}{Theorem}[section]
\newtheorem{lemma}[theorem]{Lemma}
\newtheorem{prop}[theorem]{Proposition}
\newtheorem{corollary}[theorem]{Corollary}
\newtheorem{example}[theorem]{Example}
\newtheorem{remark}[theorem]{Remark}
\newtheorem{obser}[theorem]{Observation}
\newtheorem{scholium}[theorem]{Scholium}

\title[Note as to inclusion-minimal non-Bondy systems]{Note as to inclusion-minimal non-Bondy systems}

\author[T.~J.~Kepka]{Tom\'{a}\v{s}~J.~Kepka}
\address{Faculty of Education \\
Charles University, M.\ Rettigov\'{e} 4 \\
116 39 Praha 1, Czech Republic}
\email{kepka@karlin.mff.cuni.cz}

\author[P.~C..~Nemec]{Petr~C.~N\v{e}mec}
\address{Faculty of Engineering\\
Czech University of Life Sciences, Kam\'{y}ck\'a 129 \\
165 21 Praha 6-Suchdol, Czech Republic}
\email{nemec@tf.czu.cz}

\author[J.~D.~Phillips]{J.~D.~Phillips}
\address{Department of Mathematics \& Computer Science \\
Northern Michigan University \\ Marquette, MI 49855 USA}
\email{\url{jophilli@nmu.edu}}
\urladdr{\url{http://euclid.nmu.edu/~jophilli/}}

\date{\today}

\begin{document}

\begin{abstract}

Let $S$ be a finite set, $s=|S|\ge6$. Given a non-negative integer $t$, there exists an inclusion-minimal non-Bondy system $\mathscr{A}$ of size $t$ on $S$ if and only if $s+1\le t\le2s$.
\end{abstract}



\maketitle

\section{Introduction (Bondy Systems)}

Throughout the present short note, $S$ stands for a non-empty finite set, $s = |S| \ge 1, \mathcal{P}(S)$ denotes the Boolean algebra of subsets of $S$, and any subset of $\mathcal{P}(S)$ is called a \textit{system} (of subsets of $S$). Subsets of $\mathcal{P(P}(S))$ are then called \textit{ensembles} (of systems).

Let $\mathfrak{A}$ be an ensemble of systems. A system $\mathscr{A} \in \mathfrak{A}$ is said to be
\begin{enumerate}
\item[\rm--]\textit{size-minimal (maximal)} (in $\mathfrak{A}$) if $|\mathscr{A}| \le |\mathscr{B}| \; (|\mathscr{B}| \le |\mathscr{A}|)$ for any system $\mathscr{B} \in \mathfrak{A}$;
\item[\rm--]\textit{inclusion-minimal (maximal)} if $\mathscr{A} = \mathscr{C}$ whenever $\mathscr{C} \in \mathfrak{A}$ is such that $\mathscr{C} \subseteq \mathscr{A}$ ($\mathscr{A}  \subseteq \mathscr{C}$).\end{enumerate} 

A system $\mathscr{A}$ (of subsets of $S$) is called a \textit{Bondy system} if there exists  at least one element $a$ in $S$ such that $a$ is in  $A$ whenever $A$ and $A \cup \{a\}$ are in $\mathscr{A}$.  

The following principal result was proved by J.A. Bondy in \cite{JAB}.

\begin{theorem}\label{1.2} {\rm(Bondy)} If $|\mathscr{A}| \le s (= |S|)$ then $\mathscr{A}$ is a Bondy system.\end{theorem}

Next, we collect a handful of easy and well-known observations concerning Bondy systems.

\begin{prop}\label{1.3} 
{\rm(i)} A system $\mathscr{A}$ is a Bondy system provided that either $\cap \mathscr{A} \ne \emptyset$ or $\cup \mathscr{A} \ne S$ or $\mathscr{A}$ does not generate the Boolean algebra $\mathcal{P}(S)$.\newline
{\rm(ii)} If $\mathscr{A}$ is a Bondy system and $\mathscr{B}$ is a system such that $\mathscr{B} \subseteq \mathscr{A}$ then $\mathscr{B}$ is a Bondy system, too.\newline 
{\rm(iii)} If $\mathscr{A}$ is a Bondy system then the system $\overline{\mathscr{A}} =  \{S \setminus A: A \in \mathscr{A}\}$ is Bondy as well.\newline 
{\rm(iv)} If $\mathscr{A}$ is a Bondy system then $0 \le |\mathscr{A}| \le 2^{s-1}$.\end{prop}

\begin{prop}\label{1.4} The following conditions are equivalent for a~non-negative integer $t$:
\begin{enumerate}
\item[\rm(1)]
There is at least one Bondy system $\mathscr{A}$ on $S$ such that $|\mathscr{A}| = t$.
\item[\rm(2)]
$0 \le t \le 2^{s-1}$.
\end{enumerate}\end{prop}

\begin{prop}\label{1.5} The following conditions are equivalent for a~system $\mathscr{A}$:
\begin{enumerate}
\item[\rm(1)]
$\mathscr{A}$ is a size-maximal Bondy system.
\item[\rm(2)]
$\mathscr{A}$ is an inclusion-maximal Bondy system.
\item[\rm(3)]
$\mathscr{A}$ is a Bondy system and $|\mathscr{A}| = 2^{s-1}$.
\item[\rm(4)]
There exist an element $w \in S$ and a subset $\mathscr{W} \subseteq \mathcal{P}(S \setminus \{w\})$ such that $\mathscr{A} = \mathscr{W} \cup \{A \cup \{w\} : A \in \mathcal{P}(S \setminus \{w\}) \setminus \mathscr{W}\}$.
\end{enumerate}\end{prop}

\section{Introduction (Non-Bondy Systems)}
\label{intro ii}

Let $\mathscr{A}$ be a system (on $S$). We set $\lambda(\mathscr{A}) \;(\lambda_1(\mathscr{A}), \; \mathrm{respectively}) = \{a \in S : A \cup \{a\} \in \mathscr{A}$ for at least (just, respectively) one $A \in \mathscr{A}, a \notin A\}$.

The following assertion is easy to check directly:

\begin{prop}\label{2.2} $\mathscr{A}$ is a non-Bondy system if and only if $\lambda(\mathscr{A}) = S$.\end{prop}

As usual, we continue with a bunch of straight-forward observations concerning non-Bondy systems (use \ref{1.3} and \ref{1.4}).

\begin{prop}\label{2.3} Let $\mathscr{A}$ be a non-Bondy system. Then:
\begin{enumerate}
\item[\rm(i)]
$\cap \mathscr{A} = \emptyset, \cup \mathscr{A} = S$ and $\mathscr{A}$ generates the Boolean algebra $\mathcal{P}(S)$.
\item[\rm(ii)]
A system $\mathscr{B}$ is non-Bondy, provided that $\mathscr{A} \subseteq \mathscr{B}$.
\item[\rm(iii)]
The system $\overline{\mathscr{A}}$ is non-Bondy, too.
\end{enumerate}\end{prop}

\begin{prop}\label{2.4} A system $\mathscr{A}$ is non-Bondy provided that $(2 \le)$ $2^{s-1} + 1 \le \  |\mathscr{A}| \; (\le 2^s)$.\end{prop}

\begin{theorem}\label{2.5} $|\mathscr{A}| \ge s+1$ for every non-Bondy system $\mathscr{A}$.\end{theorem}

\begin{proof} This is an equivalent form of \ref{1.2}.\end{proof}

\begin{example}\label{2.6} \rm The system $\mathscr{A} =  \{\emptyset, \{a\} : a \in S \}$ is non-Bondy and $|\mathscr{A}| = s + 1$.\end{example}

\begin{prop}\label{2.7} A non-Bondy system $\mathscr{A}$ is size-minimal if and only if $|\mathscr{A}| = s + 1$.\end{prop}

\begin{proof} Combine \ref{2.5} and \ref{2.6}.\end{proof}

\begin{prop}\label{2.8} The following conditions are equivalent for a~non-negative integer $t$:
\begin{enumerate}
\item[\rm(1)] There is at least one non-Bondy system $\mathscr{A}$ on $S$ such that $|\mathscr{A}| = t$.
\item[\rm(2)] $s + 1 \le t \le 2^s$.
\end{enumerate}\end{prop}

\begin{proof} Combine \ref{2.2}(ii) and \ref{2.5}.\end{proof}

We put $\varkappa(\mathscr{A}) = \{A \in \mathscr{A} : | A \div B| = 1$ for at least one $B \in \mathscr{A}\}$ for any system $\mathscr{A}$.

The following lemma is obvious.

\begin{lemma}\label{2.10} $\lambda(\varkappa(\mathscr{A})) = \lambda(\mathscr{A})$ and $\lambda_1(\varkappa(\mathscr{A})) = \lambda_1(\mathscr{A})$.\end{lemma}

\begin{prop}\label{2.11} Let $\mathscr{A}$ be a non-Bondy system. Then: \begin{enumerate}
\item[\rm(i)] The system $\varkappa(\mathscr{A})$ is non-Bondy. 
\item[\rm(ii)] $\varkappa(\mathscr{A}) = \mathscr{A}$, provided that the system $\mathscr{A}$ is inclusion-minimal.\end{enumerate}\end{prop}

\begin{proof} (i) By \ref{2.10} and \ref{2.2}, $\lambda(\varkappa(\mathscr{A})) = \lambda(\mathscr{A}) = S$. It remains to use \ref{2.2} once more.\newline
(ii) We have $\varkappa(\mathscr{A}) \subseteq \mathscr{A}$ and the assertion follows from (i).\end{proof}

\begin{prop}\label{2.12} Let $\mathscr{A}$ by a system (on $S$) such that $\varkappa(\mathscr{A}) = \mathscr{A}$ and $\lambda_1(\mathscr{A}) = S$. Then $\mathscr{A}$ is an inclusion-minimal non-Bondy system (on $S$).\end{prop}

\begin{proof} Clearly, $\lambda_1(\mathscr{A}) \subseteq \lambda(\mathscr{A})$, and so $\lambda(\mathscr{A}) = S$ and $\mathscr{A}$ is non-Bondy by \ref{2.2}.

Let $\mathscr{B}$ by a non-Bondy system such that $\mathscr{B} \subseteq \mathscr{A}$. Now, take any $A \in \mathscr{A}$. As $\varkappa(\mathscr{A}) = \mathscr{A}$, we find $B \in \mathscr{A}$ with $|A \div B| = 1$. Assume, firstly, that $|A| < |B|$. Then $B = A \cup \{a\}, a \in S \setminus A$. Since $\mathscr{B}$ is non-Bondy, $C \cup \{a\} \in \mathscr{B}$ for some $C \in \mathscr{B}, a \notin C$. But $a \in \lambda_1(\mathscr{A})$, and therefore $C = A, A \in \mathscr{B}$. On the other hand, if $|B| \le |A|$ then $|A| = |B| + 1, A = B \cup \{b\}, b \in S \setminus B$. Proceeding similarly as before, we find that $A \in \mathscr{B}$.

We have checked that $\mathscr{B} = \mathscr{A}$, and thus $\mathscr{A}$ is inclusion-minimal.\end{proof}

\begin{example}\label{2.13} \rm (cf. \ref{2.12}) Put $S = \{1, 2, \dots, 7\} (s = 7)$ and $ \mathscr{A} = \{A_1, A_2, A_3, A_4, B_1, B_2, C_1, C_2, D_1, D_2\}$, where $A_i = \{i\}, 1 \le i \le 4, B_1 = \{1,2\}, B_2 = \{3,4\}, C_1 = \{1,2,5\}, C_2 = \{3,4,5\}, D_1 = \{1,2,5,6\}$, and $D_2 = \{3,4,5,7\}$. We have $|\mathscr{A}| = 10$ and $\lambda(\mathscr{A}) = S$. By 2.2, $\mathscr{A}$ is a~non-Bondy system and we claim that $\mathscr{A}$ is inclusion-minimal.

For, let $\mathscr{B} \subseteq \mathscr{A}, \mathscr{B}$ non-Bondy. As $\lambda(\mathscr{B}) = S$, we see readily that either $\mathscr{A} \setminus \{C_1\} \subseteq \mathscr{B}$ or $\mathscr{A} \setminus \{C_2\} \subseteq \mathscr{B}$. If $C_1 \notin \mathscr{B}$ then $6 \notin \lambda(\mathscr{B})$, a~contradiction. Thus, $C_1 \in \mathscr{B}$ and, symmetrically, $C_2 \in \mathscr{B}$. It follows that $\mathscr{B} = \mathscr{A}$.

The system $\mathscr{A}$ is inclusion-minimal. In spite of this fact, we get $\lambda_1(\mathscr{A}) = \{1,2,3,4,6,7\} = S \setminus \{5\}$.\end{example}

Any system $\mathscr{A}$ satisfying $\varkappa(\mathscr{A}) = \mathscr{A}$ and $\lambda_1(\mathscr{A}) = S$ is an inclusion-minimal non-Bondy system. Such systems will be called \textit{slender non-Bondy systems} in what follows.

Given a positive integer $s$, we define by $\mathfrak{a}(s) (\mathfrak{b}(s)$, respectively) the set of integers $t$ such that there exists at least one inclusion-minimal (slender, respectively) non-Bondy system $\mathscr{A}, |\mathscr{A}| = t$, on an $s$-element set $S$.

\begin{prop}\label{2.16} {\rm(i)} $\mathfrak{a}(s) \subseteq \{s + 1, \dots, 2s\}$.\newline
{\rm(ii)} $\mathfrak{b}(s) \subseteq \mathfrak{a}(s)$.\end{prop}

\begin{proof} (i) Let $|S| = s$ and let $\mathscr{A}$ by a non-Bondy system on $S$. By \ref{2.5}, $t = |\mathscr{A}| \ge s + 1$. Now, if $\mathscr{A}$ were inclusion-minimal, we should have $\varkappa(\mathscr{A}) = \mathscr{A}$ (\ref{2.11}(iii)). Anyway, $\lambda(\mathscr{A}) = S$ and for each $a \in S$ we find $A_a, B_a \in \mathscr{A}$ such that $a \notin A_a$ and $B_a =  A_a \cup \{a\}$. Put  $\mathscr{B} = \{A_a, B_a : a \in S \}$. Then $\mathscr{B} \subseteq \mathscr{A}, \lambda(\mathscr{B}) = S$ and $\mathscr{B}$ is non-Bondy by 2.2. Clearly, $|\mathscr{B}| \le 2s$.\newline 
(ii) This is obvious.\end{proof}

\section{The case $1 \le s \le 6$}

\begin{obser}\label{3.1} \rm Let $S = \{1\} \; (s = 1)$. Then $\mathscr{A}_1 = \{\emptyset, \{1\}\}$ is the only non-Bondy system on $S$. We have
$\overline{\mathscr{A}_1} = \mathscr{A}_1$.\end{obser}

\begin{obser}\label{3.2} \rm Let $S = \{1,2\} \; (s = 2)$. Then, up to isomorphism,\newline
$\mathscr{A}_1 = \{\emptyset, \{1\}, \{1,2\}\}$,\newline
$\mathscr{A}_2 = \{\emptyset, \{1\}, \{2\}\}$,\newline
$\mathscr{A}_3 = \{ \{1\}, \{2\}, \{1,2\}\}$\newline
are all inclusion-minimal non-Bondy systems on $S$. We have $|\mathscr{A_i}| = 3$ and so the systems are size-minimal, in fact. Besides, $\overline{\mathscr{A}_1} = \mathscr{A}_1, \overline{\mathscr{A}_2} = \mathscr{A}_3$.\end{obser}

\begin{obser}\label{3.3} \rm Let $S = \{1,2,3\} \; (s= 3)$. Then, up to isomorphism,\newline 
$\mathscr{A}_1 = \{\emptyset, \{1\}, \{1,2\},\{1,2,3\}\}$,\newline
$\mathscr{A}_2 =\{\emptyset, \{1\}, \{1,2\}, \{1,3\}$,\newline
$\mathscr{A}_3 = \{\emptyset, \{1\}, \{2\},\{1,3\}\}$,\newline
$\mathscr{A}_4 = \{\emptyset, \{1\}, \{2\}, \{3\}\}$,\newline
$\mathscr{A}_5 = \{\{1\}, \{2\}, \{1,2\},\{1,2,3\}\}$,\newline
$\mathscr{A}_6 = \{\{1\}, \{1,3\}, \{2,3\}, \{1,2,3\}\}$,\newline
$\mathscr{A}_7 = \{ \{1,2\}, \{1,3\}, \{2,3\},\{1,2,3\}\}$,\newline
$\mathscr{A}_8 = \{\{1\}, \{2\}, \{1,2\}, \{1,3\}\}$\newline
are all inclusion-minimal non-Bondy systems on $S$. We have $| \mathscr{A}_i| = 4$ and the systems are size minimal. Besides, $\overline{\mathscr{A}_1} \cong  \mathscr{A}_1, \overline{\mathscr{A}_2} \cong  \mathscr{A}_5, \overline{\mathscr{A}_3} \cong  \mathscr{A}_6, \overline{\mathscr{A}_4} \cong  \mathscr{A}_7,$ and  $\overline{\mathscr{A}_8} \cong  \mathscr{A}_8.$\end{obser}

\begin{obser}\label{3.4} \rm Let $S = \{1,2,3,4\} \; (s = 4)$. The following systems are all inclusion-minimal non-Bondy systems:\newline
$\mathscr{A}_1 = \{\emptyset, \{1\}, \{2\}, \{3\},\{4\}\}$, $|\mathscr{A}_1|=5$,\newline
$\mathscr{A}_2 = \{\emptyset, \{1\}, \{2\}, \{1,2,3\}, \{1,2,3,4\}\}$, $|\mathscr{A}_2|=5$,\newline
$\mathscr{A}_3 = \{ \{1\}, \{2\}, \{3\},$ $\{4\}, \{1,2\}, \{3,4\}\}$, $|\mathscr{A}_3|=6$,\newline
$\mathscr{A}_4 = \{\emptyset, \{1\}, \{1,2\}, \{1,2,3\},\{1,2,3,4\}\}$, $|\mathscr{A}_4|=5$.\end{obser}

\begin{obser}\label{3.5} \rm Let $S = \{1,2,3,4,5\} \;(s = 5)$. The systems\newline
$\mathscr{A}_1 = \{\emptyset, \{1\}, \{2\}, \{3\}, \{4\}, \{5\}\}$,\newline 
$\mathscr{A}_2 = \{\emptyset, \{1\}, \{1,2\}, \{1,2,3\}, \{1,2,3,4\},\{1,2,3,4,5\}\}$,\newline
$\mathscr{A}_3 = \{\{1\}, \{2\}, \{4\}, \{5\}, \{1,2\}, \{1,3\}, \{4,5\}\}$,\newline
$\mathscr{A}_4 = \{\{1\},\{2\},\{3\},\{4\}, \{1,2\}, \{3,4\}, \{1,2,3,4\},\{1,2,3,4,5\}\}$,\newline
$\mathscr{A}_5 = \{\{1,2\}, \{2,3\},\{3,4\}, \{4,5\}, \{1,5\}, \{1,2,4\}, \{1,3,4\}, \{1,3,5\}$,\newline
$\{2,3,5\}, \{2,4,5\}\}$\newline
are inclusion-minimal non-Bondy systems, $|\mathscr{A}_1| = 6 = |\mathscr{A}_2|, |\mathscr{A}_3|= 7, |\mathscr{A}_4| = 8, |\mathscr{A}_5| = 10.$\end{obser}

\begin{obser}\label{3.6} \rm Let $S = \{1,2,3,4,5,6\} \; (s = 6)$. The systems\newline
$\mathscr{A}_1 = \{\emptyset, \{1\}, \{2\}, \{3\}, \{4\}, \{5\}, \{6\}\}$,\newline
$\mathscr{A}_2 = \{ \{1\}, \{2\}, \{1,2\}, \{1,3\}, \{4\},\{5\}, \{4,5\},\{4,5,6\}\}$,\newline
$\mathscr{A}_3 = \{\{1\}, \{2\}, \{1,2\}, \{3\}, \{4\}, \{3,4\}, \{1,2,3,4\},\{1,2,3,4,5\}, \{4,6\}\}$,\newline
$\mathscr{A}_4 =  \{\{1\}, \{2\}, \{1,2\}, \{3\}, \{4\}, \{3,4\}, \{5\}, \{5,6\}$, $\{1,2,3,4\}$,\newline
$\{1,2,3,4,5\}\}$,\newline
$\mathscr{A}_5 = \{\{1,2\}, \{2,3\}, \{3,4\}, \{4,5\}, \{1,5\}, \{1,2,4\},\{2,3,5\}, \{1,3,4\}$,\newline
$\{2,4,5\}, \{1,3,5\}, \{1,3,5,6\}\}$,\newline
$\mathscr{A}_6 = \{\emptyset, \{6\}, \{1,2\},\{2,3\}, \{3,4\}, \{4,5\}, \{1,5\}, \{1,2,4\}, \{2,3,5\}$,\newline
$\{1,3,4\},\{ 2,4,5\}, \{1,3,5\}\}$\newline
are inclusion-minimal non-Bondy systems, $|\mathscr{A}_1| = 7, |\mathscr{A}_2| = 8, |\mathscr{A}_3| = 9, |\mathscr{A}_4| = 10, |\mathscr{A}_5| =11, |\mathscr{A}_6| = 12$. We have $S \notin \mathscr{A}_i$.\end{obser}

\begin{remark}\label{3.7} \rm One can use \ref{2.12} to verify that all the systems $\mathscr{A}$ designed in \ref{3.2}, \ref{3.3}, $\dots$, \ref{3.6} are slender non-Bondy. It is easy to check that $\varkappa(\mathscr{A}) = \mathscr{A}$ and $\lambda_1(\mathscr{A}) = S$ in all the cases.\end{remark}

\begin{scholium}\label{3.8} {\rm (i)} $\mathfrak{a}(1) = \{2\} = \mathfrak{b}(1)$.\begin{enumerate} 
\item[\rm(ii)] $\mathfrak{a}(2) = \{3\} = \mathfrak{b}(2)$.
\item[\rm(iii)] $\mathfrak{a}(3) = \{4\} = \mathfrak{b}(3)$.
\item[\rm(iv)] $\mathfrak{a}(4) = \{5,6\} = \mathfrak{b}(4)$.
\item[\rm(v)] $\mathfrak{a}(5) = \{6,7,8,10\} = \mathfrak{b}(5)$.
\end{enumerate}\end{scholium}

\section{Main result}

\begin{lemma}\label{4.1} Let $S_i, i = 1, 2$ be disjoint finite sets, $s_i = |S_i| \ge 1$ and let $\mathscr{A}_i$ be inclusion-minimal non-Bondy systems on $S_i, \emptyset \notin \mathscr{A}_i$. Put $\mathscr{A} = \mathscr{A}_1 \cup \mathscr{A}_2$. Then:\begin{enumerate}
\item[\rm(i)] $\mathscr{A}$ is an inclusion-minimal non-Bondy system defined on the (disjoint) union $S = S_1 \cup S_2$.
\item[\rm(ii)] $s = |S| = s_1 + s_2$ and $\mathscr{A}| = |\mathscr{A}_1| + |\mathscr{A}_2|$.
\item[\rm(iii)] $\emptyset \notin \mathscr{A}$ and $S \notin \mathscr{A}$.
\item[\rm(iv)] $\lambda_1(\mathscr{A}) = \lambda_1(\mathscr{A}_1) \cup \lambda_1(\mathscr{A}_2)$.
\item[\rm(v)] If both $\mathscr{A}_i$ are slender then $\mathscr{A}$ is slender as well.\end{enumerate}\end{lemma}

\begin{proof} Apparently, $\lambda(\mathscr{A}) = \lambda(\mathscr{A}_1) \cup \lambda(\mathscr{A}_2) = S_1 \cup S_2 = S$, and so $\mathscr{A}$ is non-Bondy (on $S$) by \ref{2.2}. Now, if $\mathscr{B}$ were a non-Bondy system defined on $S$ such that $\mathscr{B} \subseteq \mathscr{A}$ then $\mathscr{B} = \mathscr{B}_1 \cup \mathscr{B}_2, \mathscr{B}_i = \mathscr{B} \cap \mathscr{A}_i, S = \lambda(\mathscr{B}) = \lambda(\mathscr{B}_1) \cup \lambda(\mathscr{B}_2), \lambda(\mathscr{B}_i) = S_i, \mathscr{B}_i$ are non-Bondy on $S_i, \mathscr{B}_i =  \mathscr{A}_i, \mathscr{B}= \mathscr{A}$. The rest is easy to see.\end{proof}

\begin{corollary}\label{4.2} Let $s_1, s_2$ be positive integers. Then $\mathfrak{a}(s_1) + \mathfrak{a}(s_2) \subseteq \mathfrak{a}(s_1 + s_2)$ and $\mathfrak{b}(s_1) + \mathfrak{b}(s_2) \subseteq \mathfrak{b}(s_1 + s_2)$.\end{corollary}

\begin{lemma}\label{4.3} $2s \in \mathfrak{b}(s), 5 \le s \le 9$.\end{lemma}

\begin{proof} Denote $S_{(s)} = \{1,2, \dots, s\}$ and define $\mathscr{A}_5 = \{\{1,2,\}, \{2,3\}$, $\{3,4\}, \{4,5\}, \{1,5\}, \{1,2, 4\}, \{2,3,5\}, \{1,3,4\}, \{2,4,5\}, \{1,3,5\}\}, \mathscr{A}_6 = \mathscr{A}_5 \cup \{S_{(5)}, S_{(6)}\}, \mathscr{A}_7 = \mathscr{A}_6 \cup \{\{6\}, \{6,7\}\}, \mathscr{A}_8 = \mathscr{A}_7 \cup \{\{1,7\}, \{1,7,8\}\}$ and $\mathscr{A}_9 = \mathscr{A}_8 \cup \{S_{(8)}, S_{(9)}\}$. We check easily that $\varkappa
(\mathscr{A}_s) = \mathscr{A}_s$ and $\lambda_1(\mathscr{A}_s) = S_{(s)}, 5 \le s \le 9$. Thus all five of the systems $\mathscr{A}_s$ are slender non-Bondy systems (on the respective set). We have $|\mathscr{A}_s| = 2s$.\end{proof}

\begin{lemma}\label{4.4} Let $s = |S| \ge 5$. Then there exists at least one slender non-Bondy system $\mathscr{A}$ on $S$ such that $|\mathscr{A}| = 2s$ and $\emptyset \notin \mathscr{A} (S \notin \mathscr{A}$, respectively).\end{lemma}

\begin{proof} Using the correspondence $\mathscr{A} \longleftrightarrow \overline{\mathscr{A}}$, we restrict ourselves to the case $\emptyset \notin \mathscr{A}$. If $s \le 9$ then the assertion is settled by \ref{4.3}. If $s = 5r, r \ge 2$, then the result follows from \ref{4.1} by means of induction. The remaining case is $s = 5r + t, r \ge1, 6 \le t \le 9$. Now, we combine \ref{4.1}, \ref{4.3}, and the foregoing achievements.\end{proof}

\begin{lemma}\label{4.5} Let $s = |S| \ge 1$ and let $\mathscr{A}$ be an inclusion-minimal non-Bondy system defined on $S$. Take $A \in \mathscr{A}, w \notin S$ and put $\mathscr{B} = \mathscr{A} \cup \{A \cup \{w\}\}$ and $T = S \cup \{w\}$. Then:\begin{enumerate}
\item[\rm(i)] $\mathscr{B}$ is an inclusion-minimal non-Bondy system on $T, |\mathscr{B}| = |\mathscr{A}| +1, |T| = s + 1$.
\item[\rm(ii)] $\mathscr{B}$ is slender, provided that $\mathscr{A}$ is such.
\item[\rm(iii)] $\emptyset \notin \mathscr{B}$ provided that $\emptyset \notin \mathscr{A}$.
\item[\rm(iv)] If $A \ne S$ then $T \notin \mathscr{B}$.\end{enumerate}\end{lemma}

\begin{proof} (i) Clearly, $\lambda(\mathscr{B}) = \lambda(\mathscr{A}) \cup \{w\} = S \cup \{w\} = T$. Thus $\mathscr{B}$ is non-Bondy on $T$ (\ref{2.2}). Furthermore, let $\mathscr{C} \subseteq \mathscr{B}, \mathscr{C}$ non-Bondy on $T$. As $\lambda(\mathscr{C}) = T$, we have $A \cup \{w\} \in \mathscr{C}$. Put $\mathscr{D} = \mathscr{C} \setminus \{A \cup \{w\}\}$. Clearly $\lambda(\mathscr{D}) = \lambda(\mathscr{C}) \setminus \{w\} = S$, and therefore $\mathscr{D}$ is a non-Bondy system on $S, \mathscr{D} \subseteq \mathscr{A}$. Since $\mathscr{A}$ was inclusion-minimal, we conclude that $\mathscr{D} = \mathscr{A}$ and $\mathscr{C} = \mathscr{B}$.\newline
(ii) We have $\lambda_1(\mathscr{B}) = \lambda_1(\mathscr{A}) \cup \{w\} = S \cup \{w\} = T$.\newline
(iii) and (iv) This is obvious.\end{proof}

\begin{theorem}\label{4.6} Let $S$ be a finite set such that $s = |S| \ge 6$. The following conditions are equivalent for a non-negative integer $t$:\begin{enumerate}
\item[\rm(1)] $(7 \le) \; s + 1 \le t \le 2s$.
\item[\rm(2)] There exists at least one inclusion-minimal non-Bondy system $\mathscr{A}$ on $S$ such that $|\mathscr{A}| = t$.
\item[\rm(3)] There exists at least one slender non-Bondy system $\mathscr{A}$ on $S$ such that $\emptyset \notin \mathscr{A}$ and $|A|=t$.
\item[\rm(4)] There exists at least one slender non-Bondy system $\mathscr{A}$ on $S$ such that $S \notin \mathscr{A}$ and $|A|=t$.\end{enumerate}\end{theorem}

\begin{proof} (1) implies (4). The case $s = 6$ is treated in \ref{3.6} and, using \ref{4.5} and \ref{4.4}, we can proceed by induction on $s$. \newline
(4) implies (3). Use the correspondence $\mathscr{A} \longleftrightarrow \overline{\mathscr{A}}$. \newline
(3) implies (2). This implication is trivial. \newline
(2) implies (1). It follows from \ref{2.16}(i).\end{proof}

\end{document}